\documentclass{amsart}
\usepackage{amsmath,amssymb}
\usepackage[all]{xy}
\SelectTips{cm}{}

\theoremstyle{plain}
\newtheorem{introthm}{Theorem}

\newtheorem{theorem}{Theorem}[section]
\newtheorem{proposition}[theorem]{Proposition}
\newtheorem{lemma}[theorem]{Lemma}
\newtheorem{corollary}[theorem]{Corollary}

\newtheorem{problem}[theorem]{Problem}

\newtheorem{theorem*}{Theorem}[]

\theoremstyle{definition}

\newtheorem{example}[theorem]{Example}

\theoremstyle{remark}
\newtheorem{remark}[theorem]{Remark}
\newtheorem{remarks}[theorem]{Remarks}

\newcommand{\lin }{\mathcal{H}om }

\def\:{{\colon}}

\def\Q{{\mathbb Q}}
\def\C{{\mathbb C}}

\def\D{{\mathcal{D}}}

\def\G{\mathcal{G}}
\def\B{\mathcal{B}}
\def\CC{{\mathcal C}}

\def\map{\mathrm{map}}

\def\aut{\mathrm{aut}}

\def\Baut{\mathrm{Baut}}
\def\aut{\mathrm{aut}}
\def\der{\mathrm{Der}}
\def\dder{\mathcal{D}\mathrm{er}}
\def\Der{\mathrm{Der}}

\def\Hom{\mathrm{Hom}}

\def\cat0{\mathrm{cat}_0}

\def\ev{\mathrm{ev}}
\def\ad{\mathrm{ad}}

\def\cat{{\mathrm {cat}}}

\def\ev{{\mathrm{ev}}}

\def\id{\mathrm{id}}

\begin{document}

\title[The Classifying Space for Fibrewise Self-Equivalences]
{Rational Homotopy Type  of  the Classifying Space for Fibrewise Self-Equivalences}

\author{Urtzi Buijs } \thanks{U. Buijs was partially supported by the {\em Ministerio de Ciencia e  
Innovaci\'on}  grant MTM2010-15831 and by the
  {\em Junta de Andaluc\'\i a} grant FQM-213.}
\address{Departament d'\`Algebra i Geometria, Universitat de Barcelona.
Gran Via de les Corts Catalanes, 585, 08007 Barcelona, Spain}

\email{ubuijs@ub.edu}

\author{Samuel B. Smith}
\address{Department of Mathematics,
  Saint Joseph's University,
  Philadelphia, PA 19131 U.S.A.}
\email{smith@sju.edu}

\date{\today}



\begin{abstract}
Let $p \colon E \to B$ be a fibration of simply connected CW complexes with finite base $B$ and fibre $F$.  Let $\aut_1(p)$ denote
the identity component of the space  of  all  fibre-homotopy self-equivalences of $p$.  Let  $\Baut_1(p)$ denote the classifying space for this topological monoid.
We give a differential graded Lie algebra model for $\Baut_1(p)$,   connecting  recent work in \cite{FLS} and \cite{BFM}.
We use this model to give   classification results for    the rational homotopy types represented by    $\Baut_1(p)$  and  also  to obtain  conditions under which the monoid $\aut_1(p)$ is a double loop-space after rationalization.

  \end{abstract}

\maketitle

\section{Introduction}\label{sec:intro}
We consider  a  variation on the following classical problem for principal bundles.  Fix a space   $B$  and   group $G$.  Given a principal $G$-bundle $p \colon E \to B$ let $\G(p)$ denote the {\em gauge group} consisting  of $G$-equivariant automorphisms of $p$ as studied in  \cite{AB}.  The gauge group classification problem is that of determining the distinct H-homotopy types represented by the   groups $\G(p)$ for fixed $B$ and $G$ (see, for example,  \cite{K, CS, KT}).  Refining this  problem  is that   of determining the homotopy types represented   by the classifying spaces $\mathrm{B}\G(p)$  (see, for example, \cite{CS, KT, Tsu}).    In this connection, we have the    identity   $\ \mathrm{B}\G(p) \simeq \map(B, BG; h)$ where $h \colon B \to BG$ is the classifying map   \cite{Got, AB}.

The complexity of the gauge group classification problem is     strictly a torsion phenomenon.
After rationalization, the gauge groups  are all H-homotopy commutative and pairwise H-homotopy equivalent  when  $B$ is a compact metric space and $G$        is   a connected Lie group \cite[Th.D]{KSS}.    Further,   the  classifying spaces $B\G(p)$    are pairwise  homotopy equivalent   (see Remark  \ref{rems}.3, below).

A  related classification problem with  rich  rational homotopy theory   is obtained by moving to the setting  of  fibrations  $p \colon E \to B$ with  fibre $F.$  Here the analogue of the gauge group is the monoid $\aut_1(p)$ of fibrewise self-equivalences homotopic to the identity.     Passing to  the Dold-Lashof classifying space  \cite{DL}, for $B$ simply connected, we have the identity \begin{equation} \label{eq:Baut}
\Baut_1(p)  \simeq \widetilde{\map}(B, \Baut_1(F); h). \end{equation} Here   $h \colon B \to \Baut_1(F)$ is the classifying map of the fibration $p$ and $\widetilde{\map}$ denotes the universal cover of the function space by \cite{BHMP} (see Proposition \ref{pro:BHMP}, below).    Unlike the gauge groups,    the monoids $\aut_1(p)$  and their classifying spaces $\Baut_1(p$)  display  non-equivalent  rational  homotopy types for fixed $F$ and $B.$    This is true even in the case of principal bundles.  For instance, we have:

\begin{example} \label{example1} The classifying spaces $\Baut_1(p)$ for    principal $SU(n)$-bundles $p$ over $\C P^m$  represent $\mathrm{max}\{1,n-1, m\}$ distinct rational homotopy types.

\end{example}

The monoid $\aut_1(p)$ is     not  often H-homotopy commutative after rationalization.  For example,
  when $p$ is a principal $SU(n)$-bundle over $\C P^{m},$  as above, $\aut_1(p)$   is never H-homotopy commutative after rationalization   for $n \geq 4$. On the other hand, there are a number of interesting cases in which  $\aut_1(p)$ is, in fact,  H-equivalent to a double loop-space  after rationalization, i.e., $\Baut_1(p)$ is a rational H-space.   We mention one here:

\begin{example} \label{example2}  Let  $p \colon E \to B$ be a principal $G$-bundle with decomposable rational characteristic classes where $G$ is a simply  connected Lie group and $B$ is  a finite, simply connected CW complex.    Let   $J   = \mathrm{Hur}_E \circ j_\sharp \colon \pi_*(G) \to H_*(E; \Q)$ where
$j \colon G \to E$ is the fibre inclusion,    $j_\sharp \colon \pi_*(G) \to \pi_*(E) $  is  the induced map and  $\mathrm{Hur}_E \colon \pi_*(E)  \to H_*(E; \Q)$ is the rational Hurewicz map.   If $E$ is formal and  $J= 0$ then   $\Baut_1(p)$  is a rational H-space.
\end{example}

We deduce Examples \ref{example1} and \ref{example2} and related results as a consequence of   a differential graded (DG) Lie algebra model  or {\em Quillen model} for the space $\Baut_1(p)$.    Our starting point  is  the description of   the rational Samelson Lie algebra of $\aut_1(p)$ given in \cite{FLS} which, in turn,
extends the famous model for $\Baut_1(F)$ described by Sullivan \cite{Su}.    Let $p \: E \to B$  be a fibration of  simply connected CW complexes.         Then $p$  admits a relative minimal model which is an injection of DG
algebras $I \: \B,d_B \to \B \otimes \Lambda W,D$    (see \cite[Pro.15.6]{FHT}).
Let  $\der^n_{\B}(\B \otimes \Lambda W )$ denote  the vector  space of all linear  self maps $\theta$ of $\B \otimes \Lambda W$ reducing degrees by $n$ and satisfying $\theta(ab)= \theta(a)b - (-1)^{n|a|}a\theta(b)$ and  $\theta(\B)=0.$  The graded space  $\der_{\B}(\B \otimes \Lambda W )$  then has  the structure
of a DG Lie algebra with the commutator bracket $[\theta_1, \theta_2]
= \theta_1 \circ \theta_2 - (-1)^{|\theta_1||\theta_2|}\theta_2
\circ \theta_1$ and differential $\D(\theta) = [D, \theta].$
When $E$ and $B$ are finite,  \cite[Th.1]{FLS} gives an isomorphism of graded Lie algebras
\begin{equation} \label{eq:Samelson} \pi_*(\aut_1(p) ) \otimes  \Q,  \, \, [ \, , \, ]  \cong \, H_*(\der_{\B}(\B \otimes \Lambda W)), \, \,  [ \, , \,]
\end{equation}
where the former space has the Samelson bracket. 
We refine  the above isomorphism  with our main result. 

Define the connected DG Lie algebra $\dder_{\B}(\B \otimes \Lambda W ), \D$ by:  $$\dder^n_{\B}(\B \otimes \Lambda W )= \left\lbrace
           \begin{array}{c l}
              \der^n_{\B}(\B \otimes \Lambda W ) & \text{for $n\geq 2$},\\
              Z\der^1_{\B}(\B \otimes \Lambda W ) & \text{for $n=1$}.
           \end{array}
         \right.$$
Here $Z$ means the space of
 cycles.  The bracket and differential $\D$ are those described above.

\begin{introthm} \label{thm:main}
Let $p \colon E \to B$ be a fibration of  simply connected finite type CW complexes with fibre $F$.  Suppose $B$ and $F$ are   finite.  Then
$\dder_{\B}(\B \otimes \Lambda W), \D$ is a  Quillen  model for $\Baut_1(p).$
\end{introthm}

The paper is organized as follows.   In Section \ref{proof},   we prove Theorem \ref{thm:main}.
Our proof    is based on the main result of \cite{BFM} which gives a Quillen model for function space components.    We  construct an explicit   isomorphism between the  DG Lie algebra $\dder_\B(\B \otimes \Lambda W), \D$ and this model  for  the   component $\map(B,   \Baut_1(F); h)$. We then invoke the identity (\ref{eq:Baut}).
 In   Section \ref{problems},  we  apply Theorem \ref{thm:main} to deduce a number of applications    along the lines  of Examples \ref{example1} and \ref{example2}.  We assume familiarity with rational homotopy theory as in the text   \cite{FHT}.

\section{A Quillen Model for $\Baut_1(p)$} \label{proof}
Fix a fibration  $p \colon E \to B$ with fibre $F$  as in Theorem \ref{thm:main}.   That is, all spaces are simply connected, finite type CW complexes with $B$ and $F$   finite.    Define  $\aut_1(p)$ to be  the path-component   of the identity map in  the  fibrewise function space $\map_B(E, E)$ consisting of    maps $f \colon E \to E$  satisfying $p \circ f = p.$ Here $\map_B(E, E)$ is topologized as a subspace of $\map(E, E).$  Then, with our hypotheses,    $\aut_1(p)$ is a strictly associative, homotopy inversive  H-space of CW type (see \cite[Pro.2.3]{FLS}) and so admits a Dold-Lashof classifying space $\Baut_1(p)$ \cite{DL}.   The space $\Baut_1(p)$ is, in turn,  a
simply connected CW complex and so admits a rationalization $ \Baut_1(p)_\Q.$
Our main result, Theorem \ref{thm:main}, determines the homotopy type of $ \Baut_1(p)_\Q.$

Our proof of Theorem \ref{thm:main}  depends on three results which we discuss now.   First off, we deduce the identity (\ref{eq:Baut}) as   a consequence of the main result of \cite{BHMP}.   The fibration $p$ above is induced from the universal $F$-fibration by a map $f \colon B \to \Baut(F)$ (cf. \cite{May}). Since $B$ is simply connected, $f$ lifts to a map $h \colon B \to \Baut_1(F)$.  We refer to $h$ as the {\em classifying map} for $p$. 
\begin{proposition} \label{pro:BHMP}  \emph{(Booth-Heath-Morgan-Piccinni)} Let $p \colon E \to B$ be a fibration with fibre $F$ as hypothesized in Theorem \ref{thm:main}.  Then there is a homotopy equivalence:
$$
\Baut_1(p)  \simeq \widetilde{\map}(B, \Baut_1(F); h).$$
\end{proposition}
 \begin{proof}
  By \cite[Th.3.3]{BHMP}, there is a homotopy equivalence $$
\Baut(p)  \simeq  \map(B, \Baut(F); f).$$  Since $B$ is simply connected,  the inclusion $\aut_1(F) \hookrightarrow \aut(F)$ induces a homotopy equivalence $\map(B, \Baut_1(F); h) \simeq \map(B, \Baut(F); f).$  Thus $$\Baut(p)  \simeq  \map(B, \Baut_1(F); h).$$ Taking universal covers gives the desired result.   
 \end{proof}
 
 The second two results we need are in rational homotopy theory. We say a DG Lie algebra  $L, \delta$ is a {\em Quillen model} for a simply connected CW complex  $X$ if there is a quasi-isomorphism
$\psi \colon \CC^*(L, \delta) \to A_{\mathrm{PL}}(X; \Q)$.    Here $A_{\mathrm{PL}}(X; \Q)$ is the DG algebra of   rational polynomial forms on $X$ \cite[Sec.10]{FHT}.  The DG algebra $\CC^*(L, \delta)$ is the {\em commutative cochain algebra} for $L, \delta$ (see \cite[Sec.23]{FHT}).  
We recall, as graded   algebra  $\CC^*(L, \delta) = \Lambda (s^{-1}L^{\sharp}),$    the free graded algebra generated by the desuspension of the dual   $L^\sharp = \Hom(L, \Q)$ of $L$.  The differential $d$ is given by $d= d_1 + d_2$ where
$$ \langle d_1s^{-1}z; sx\rangle = - \langle z, \delta x \rangle \ \ \hbox{\ \ and \ \ } \ \ \langle d_2s^{-1}z; sx_1, sx_2 \rangle = (-1)^{|x_1|} \langle z, [x_1, x_2] \rangle.$$
Here $z \in L^\sharp$ and $ x, x_1, x_2 \in L.$

\begin{proposition}{\em (Sullivan)} \label{pro:Sull} Let $F$ be a finite, simply connected CW complex. 
Let $\Lambda W , d $ be the Sullivan minimal   model for $F$. The DG Lie
algebra $\dder(\Lambda W ), \D$ is then a Quillen model for
$\Baut_1(F)$.\hfill $\square $
\end{proposition}
\begin{proof}    Corollary 7.4.4 of \cite{Ta} gives   a Quillen model,  for $\Baut_1(F).$ By \cite[Th.2]{Gat}, 
this model is quasi-isomorphic to $\dder(\Lambda W ), \D.$
\end{proof}

The third result we need is a main result of \cite{BFM}. Let $h\colon B \to Y$ be a map of simply connected CW complexes with $B$
finite.  We say $\B, d_B$ is a {\em DG algebra model} for $B$ if there is a quasi-isomorphism  from the Sullivan minimal model of $B$ to   $\B, d_B.$  Since $B$ is finite, $B$ admits a  finite-dimensional DG algebra model $\B, d_B$.  Let $L, \delta$ be a Quillen model for
$Y$. Then we can choose a  DG algebra morphism $\phi \colon
\mathcal{C}^*(L, \delta)\to \B, d_B$ modeling   $h$. Let $\CC_*(L, \delta)$ and $\B^\sharp, d_B^\sharp$ denote the DG coalgebras dual to $\CC^*(L, \delta)$ and $\B, d_B$, respectively. 
Let  $\tilde{\phi }$ denote the degree $-1$ linear map given by the composition: 
\begin{equation} \label{eq:phi} \xymatrix{\tilde{\phi }\colon \B^\sharp \ar[r]^-{\phi^\sharp } & \mathcal{C}_*(L, \delta)=\Lambda sL\ar[r]^{\ \ \ \ \ \  \xi } & sL \ar[r]^{s}  &L} \end{equation}
where $\xi$ is the projection onto the indecomposables and $s$ is the suspension isomorphism.   

Next, let $\Hom^n(\B^\sharp ,L)$ denote  the space of linear maps lowering degrees by $n$ and $\Hom(\B^\sharp, L)$ the corresponding graded vector space. 
We equip this graded space with  a bracket $$\{ \theta_1, \theta_2 \}=[ \, , \,]\circ (\theta_1 \otimes \theta_2)\circ \Delta, $$ where $\Delta$ is the diagonal map and $[ \, , \, ]$ is the bracket in $L$. 
We also have the perturbed differential:
$$D_\phi \theta=  \delta \theta -(-1)^{|\theta|}\theta d_B^\sharp + \{\tilde{\phi },\theta\}.$$
Truncating as above,  gives  a  connected DG  Lie algebra
$\lin(\B^\sharp,L), D_\phi$ defined by
$$\lin^n(\B^\sharp ,L)= \left\lbrace
           \begin{array}{c l}
              \Hom^n(\B^\sharp,L) & \text{for $n>1$},\\
              Z\Hom^1(\B^\sharp ,L),& \text{for $n=1$}.
           \end{array}
         \right.$$
We have: 

 \begin{proposition} {\em \cite[Th.10]{BFM} and \cite[Th.7]{BFM2}}    \label{pro:BFM}
Let $B$ and $Y$  be simply connected CW complexes with $B$ finite. Let $\B, d_B$  and $L, \delta$ be a finite-dimensional DG algebra model for $B$ and a Quillen model for $Y$, respectively. Then the DG Lie algebra  $\lin(\B^\sharp, L), D_\phi$  
is a Quillen model for $\widetilde{\map }(B,Y;h).$    \qed \end{proposition}

Returning to the    fibration   $p \: E \to B$   with hypotheses as in Theorem \ref{thm:main},    let $I \: \B,d_B \to  \B \otimes \Lambda W,D$ be a  relative minimal model.    Our main result in this section is an   isomorphism of DG Lie algebras:
$$ \Psi \colon \dder_\B(\B \otimes \Lambda W), \D \to \lin(\B^\sharp,  \lin(W, \Lambda W)), D_\varphi.$$
Here $$\varphi \colon \CC^*(\dder(\Lambda W), \D) \to \B, d_B$$ is a model for the classifying map $h \colon B \to \Baut_1(F)$.  

We define  $\Psi$ to be  
  the restriction to ``universal covers" of an isomorphism of graded spaces
  $$ \Phi \colon \Hom (W
,\B\otimes \Lambda W )\stackrel{\cong}{\longrightarrow }\Hom (\B^\sharp
, \Hom (W ,\Lambda W ))$$ given by \begin{equation} \label{eq:Phi} \Phi (\theta )(\beta )(w
)=(-1)^{|\theta ||\beta |}(\beta \otimes \id_{\Lambda W })\theta(w). \end{equation} Notice that, as graded spaces, $$\der_{\B}(\B \otimes \Lambda W) = \Hom(W, \B\otimes \Lambda W)$$  and
$$  \Hom(\B^\sharp, \Der(\Lambda W)) = \Hom (\B^\sharp
, \Hom (W ,\Lambda W ))$$
with the same identities holding for universal covers.  

An   alternative description of $\Phi$ is the following: Let $\theta \in \Hom(W, B\otimes \Lambda W)$.
 Then  $\Phi (\theta) $ is   the unique map in $\Hom (\B^\sharp
, \Hom (W ,\Lambda W ))$ making the following diagram commutative:
\begin{equation}\label{diag:iso}
\xymatrix{\B^\sharp \otimes W \ar[d]_{\id  \otimes \theta }\ar[rr]^(.4){\Phi (\theta) \otimes \id }&&\Hom(W,\Lambda W)\otimes W \ar[d]^{\ev }\\
\B^\sharp \otimes \B\otimes \Lambda W \ar[rr]_(.55){\ev\otimes \id }&&\Lambda W. }
\end{equation}

We obtain a related commutative diagram  involving  the map   $$\tilde{\varphi} \colon  \colon \B^\sharp \to \dder(\Lambda W)$$   induced by the classifying map $h.$  To do so, we observe that $\tilde\varphi$  gives rise to a   particular choice of a relative minimal model $I \colon \B, d_B \to \B \otimes \Lambda W, D$ for $p \colon E \to B$. 

Note that $\tilde{\varphi}$   restricts to give a degree $-1$ element in $\Hom (\B_+
^\sharp , \Hom(W, \Lambda W))$ where $\B_+$ means elements of positive degree.    The differential $D$ in   a  relative minimal model  for $p$ decomposes as $$D=d_W+d_T,$$ where $d_W(W)\subseteq \Lambda^{\geq 2} W$ is the differential from  the minimal model of the fibre, and   $d_T(W)\subseteq \B_+\otimes \Lambda W$ is the  ``twisted'' part  of the differential.
Then $d_T$ represents a degree $-1$ element  in  
in $\Hom (W, \B_+\otimes \Lambda W ).$

\begin{lemma}\label{iso} The restriction
$$\Phi \colon \Hom (W
,\B_+\otimes \Lambda W ) \longrightarrow  \Hom
(\B_+^\sharp , \Hom (W ,\Lambda W ))$$ is an isomorphism. Given a classifying map $h \colon B \to \Baut_1(F)$,  we may choose the twisted part of the differential $d_T$ in a relative minimal model $I \colon \B, d_B \to \B \otimes \Lambda W, D$for $p \colon E \to B$  satisfying
$\Phi (d_T)= \tilde{\varphi} $. Conversely, given a relative model for $p$  we may choose a    model  $\varphi$ for $h$ with  $\Phi (d_T)= \tilde{\varphi}.$
\end{lemma}
\begin{proof}
That $\Phi$ restricts to an isomorphism is direct. Given a simply connected DG coalgebra $A, \delta$  let $\mathcal{L}^*(A, \delta)$   denote the  ``Quillen construction'',  a connected DG Lie algebra
(see \cite[Sec.22(e)]{FHT} for details).  A  DG Lie model for $h$ is then a map $ \mathcal{L}(h) \colon \mathcal{L}^*(\B^\sharp, d_B^\sharp) \to \dder(\Lambda W), \D.$  
The map $\tilde\varphi$ corresponds here to   $\mathcal{L}(h)$ restricted to the basis $\B^\sharp.$ Next we have a  quasi-isomorphism $\psi \colon \dder(\Lambda W) \to \mathrm{cl}(\mathcal{L}^*((\Lambda W)^\sharp))$
where $\mathrm{cl}(\mathcal{L}^*((\Lambda W)^\sharp)$ is the DG Lie model for  $\Baut_1(F)$ given in \cite[Ch.7]{Ta} and $\psi$ is the DG Lie algebra map constructed in \cite[Th.1]{Gat}.   The composition $\psi \circ  \mathcal{L}(h) \colon \mathcal{L}^*(\B^\sharp, d_B^\sharp) \to \mathrm{cl}(\mathcal{L}^*((\Lambda W)^\sharp), d_W^\sharp)$ gives rise, via the classification theorem \cite[Pro.7.2(11)]{Ta}, to a split exact sequence of DG Lie algebras   of the form 
$$0 \to \mathcal{L}^*((\Lambda W)^\sharp), d_W^\sharp) \to L, \delta  \to \mathcal{L}^*(\B^\sharp, d_B^\sharp) \to 0.$$ 
Applying $\CC^*$ we see $\CC^*(L, \delta) \cong \B \otimes \Lambda W, D$ for some differential $D$.  We take this as the relative model $ I \colon \B, d_B \to \B \otimes \Lambda W, D$ for $p$.  Tracing through the construction, we see
$\Phi(d_T) = \tilde\varphi.$  The converse result is proved similarly. 
\end{proof}

 Lemma \ref{iso} gives the commutativity of the following diagram:
\begin{equation}\label{diag:h}
\xymatrix{\B_+^\sharp \otimes W \ar[d]_{\id \otimes d_T }\ar[rr]^(.42){\tilde\varphi \otimes \id }&&\Hom(W,\Lambda W)\otimes W \ar[d]^{\ev }\\
\B_+^\sharp \otimes \B_+\otimes \Lambda W \ar[rr]_(.6){\ev\otimes \id }&&\Lambda W. }
\end{equation}

Our main technical result in this section is the following. 

\begin{theorem} \label{thm:quasi}  Let $\B, d_B$ be finite-dimensional DG algebra and $W$ a graded vector space of finite type.  Let $I \colon \B, d_B \to \B \otimes \Lambda W,D$ be a relative minimal model with twisted differential $d_T$.  Suppose $\tilde\varphi \colon
\B^\sharp \to \Hom(W, \Lambda W)$ gives the commutativity of diagram \emph{(\ref{diag:h})}.  The restriction 
$$\Psi \colon \dder_\B(\B\otimes \Lambda W ), \D \to \lin(\B ^\sharp ,
\dder (\Lambda W )), D_\varphi$$ of the map $\Phi$   defined by \emph{(\ref{eq:Phi})}  is an isomorphism of  graded DG Lie algebras.  
\end{theorem}
\begin{proof} 
Since  $\Psi$ is evidently an  isomorphism of graded spaces, it suffices to show $\Psi$   commutes with differentials and brackets. 

We first show that $$D_\varphi\circ \Psi
=\Psi \circ \D.$$ Let $\theta \in \der_\B(\B\otimes \Lambda W
)$, $\beta \in \B^\sharp $, $w\in W $. Then we have
\begin{eqnarray*}
\Psi (\D \theta )(\beta )(w)&=&(-1)^{(|\theta |+1)|\beta |}(\beta \otimes \id_{\Lambda W }
)(D\circ \theta -(-1)^{|\theta |} \theta \circ D)(w)\\
&=&(-1)^{(|\theta |+1)|\beta |}(\beta \otimes id_{\Lambda W })[(d_{B }\otimes id_{\Lambda W })\theta (w)+ (id_\B \otimes d_T )\theta (w)\\
&+ &(id_\B \otimes d_W )\theta (w)+(-1)^{|\theta |+1} \theta d_T w +(-1)^{|\theta |+1}  \theta d_W w ]\\
&=&(a)+(b')+(c')+(b'')+(c'').
\end{eqnarray*}
We see $(b)=(b')+(b'')$ and $(c)=(c')+(c'')$.

On the other hand,
\begin{eqnarray*}
D_\varphi \Psi(\theta) (\beta )(w) &=&(\D\circ \Psi(\theta) +(-1)^{|\theta |+1}
\Psi (\theta)
\circ d_B^\sharp + \ad_\varphi\Psi (\theta) )(\beta )(w)\\
&=&(C)+(A)+(B).
\end{eqnarray*}
A straightforward computation using the commutativity of the following diagrams
\begin{equation}\label{diag:m and d}
\xymatrix{\B^\sharp \otimes \B \otimes \B \ar[d]_{\Delta \otimes
\id } \ar[rr]^(.55){\id \otimes m }&&\B
^\sharp \otimes
\B \ar[d]^{\ev }&&\B\otimes \B^\sharp \ar[d]_{\id \otimes d_B^\sharp }\ar[r]^{d_B\otimes \id }&\B\otimes \B^\sharp \ar[d]^{\ev} \\
\B ^\sharp \otimes \B ^\sharp \otimes \B \otimes \B
\ar[rr]_(.6){\ev }&&\mathbb{Q}&&\B\otimes \B^\sharp \ar[r]_(.55){\ev }&\mathbb{Q}, }
\end{equation}
and the diagrams (\ref{diag:iso}), (\ref{diag:h}),
gives $(a)=(A), (b)=(B), (c)=(C)$.

For instance, to see explicitly the equality $(b)=(B)$, it is enough to show that the following compositions agree
$$(\ev \otimes \id )\circ (\id_{\B^\sharp }\otimes (m\otimes \id_{\Lambda W})\circ ((\id_{\B^\sharp} \otimes d_T)\circ \theta +(-1)^{|\theta |}\theta \circ d_T))\colon \B^\sharp \otimes W\to \Lambda W,$$
$$\ev \circ (\ad_\varphi \Psi( \theta) \otimes \id_{\Lambda W})\colon \B^\sharp \otimes W\to \Lambda W.$$

Note that the first composition splits into two summands. Recall that $\ad_\varphi (f)=[ \, , \,]\circ (\varphi\otimes f )\circ \Delta $, where the bracket $[ \, , \, ]$ in $ \der (\Lambda W)$  is  the commutator.    Hence the second composition also splits into two summands. Then, the first summand of both compositions appears as the two possible exterior compositions from the top left corner to the bottom right corner of the following commutative diagram, so they agree.

{\scriptsize
$$ 
\xymatrix@C=0pt @W=0pt @H=0pt{ \B^\sharp \otimes W\ar[d]_{\Delta \otimes \id}\ar[r]^(.45){\id \otimes d_T}&\B^\sharp \otimes \B\otimes \Lambda W\ar[d]^{\Delta \otimes \id }\ar[rr]^{\id \otimes \theta }&&\B^\sharp \otimes \B\otimes \B\otimes \Lambda W \ar[d]_{\Delta \otimes \id }\ar[r]^{\id \otimes m\otimes \id}&\B^\sharp \otimes \B \otimes \Lambda W \ar[d]^{\ev \otimes \id }
%\ar@{}[ld]|-{\hbox{(iii)}} 
\\
\B^\sharp \otimes \B^\sharp \otimes W \ar@{}[rd]|-{\hbox{(ii)}}\ar[d]_{\id \otimes \tilde\varphi\otimes \id } \ar[r]^(.45){\id \otimes d_T}&\B^\sharp \otimes \B^\sharp \otimes \B \otimes \Lambda W\ar[d]^{\id \otimes \ev \otimes \id } \ar[rr]^{\id \otimes \theta }&&\B^\sharp \otimes \B^\sharp \otimes \B \otimes \B \otimes \Lambda W\ar[d]_{\ev \otimes \id } \ar[r]^(.75){\ev \otimes \id }& \Lambda W \ar@{=}[d]\\
\B^\sharp \otimes \Hom (W,\Lambda W)\otimes W\ar[d]_{\Psi (\theta) \otimes \id } \ar[r]^(.65){\id \otimes \ev }& \B^\sharp \otimes \Lambda W \ar@{}[rd]|-{\hbox{(i)}}\ar[d]_{\Psi(\theta) \otimes \id }\ar[r]^(.45){\id \otimes \theta }&\B^\sharp \otimes \B \otimes \Lambda W \ar[d]^{\ev \otimes \id }\ar[r]^(.6){\ev \otimes \id }&\Lambda W \ar@{=}[d]\ar@{=}[r]&\Lambda W \ar@{=}[d]\\
\Hom (W,\Lambda W)^{\otimes 2}\otimes W\ar[r]^{\id \otimes \ev }&\Hom (W,\Lambda W)\otimes \Lambda W\ar[r]^(.7){\ev}&\Lambda W\ar@{=}[r]&\Lambda W\ar@{=}[r]&\Lambda W. }
$$
}

Interchanging $d_T$ with $\theta $ and $\tilde\varphi $ with $\Psi( \theta) $ and  
interchanging (i) with (ii) in the above diagram, we obtain the equality of the second summands of both compositions. 

To see that $\Psi$ preserves brackets,  let $\theta_1, \theta_2 \in \Der_\B(  \B \otimes \Lambda W)$ be derivations.  Recall that $[\ ,\ ]$ is the bracket in $\Der^*_{\B}(\B \otimes
\Lambda W )$ and $\{\ ,\ \}$ is the bracket in
$\mathcal{H}om(\B^\sharp , \Der (\Lambda W))$. Then, given $\beta
\in \B^\sharp $ and $w\in W$, both $\{ \Psi (\theta_1 ),\Psi
(\theta_2 )\}(\beta )(w)$ and $[\Psi (\theta _1), \Psi (\theta_2
)](\beta )(w)$ split into two summands. We write the first
summands arising from $\{ \Psi (\theta_1 ),\Psi (\theta_2 )\}$ and
$[\Psi (\theta _1), \Psi (\theta_2 )]$ as the  upper and lower  
compositions, respectively, in the following diagram:

{\scriptsize
$$  \xymatrix{&&&\B^\sharp \otimes \B^\sharp \otimes \B\otimes \B \otimes \Lambda W\ar[rd]^{\ev \otimes \id}&\\
\B^\sharp \otimes W\ar[r]^(.4){\id \otimes \theta_2}& \B^\sharp
\otimes \B \otimes \Lambda W\ar[r]^(.45){\id \otimes
\theta_1}&\B^\sharp \otimes \B\otimes \B \otimes \Lambda
W\ar[ru]^{\Delta \otimes \id }\ar[rd]_{\id \otimes m}&&\Lambda W.\\
&&&\B^\sharp \otimes \B\otimes \Lambda W\ar[ru]_{\ev \otimes \id}&
}$$} 
The compositions agree by the commutativity of the first diagram in (\ref{diag:m and d}). 
Similarly,  the second summands in these brackets agree as may be seen  by interchanging $\theta_1$ and $\theta_2$. 

A  proof,  given in terms of   elements,  goes as follows. Here we work modulo signs and summations for convenience. 
Suppose $\theta_1(w)=b_1w_1$, $\theta_2(w)=b_2w_2$,
$\theta_1(w_2)=b_1'w_1'$ and $\theta_2(w_1)=b_2'w_2'$. Then, on
the one hand
\begin{eqnarray*}
\{ \Psi (\theta_1),\Psi (\theta_2)\}(\beta )(w)&=&\sum_j[ \Psi
(\theta_1)(\beta_j'),\Psi (\theta_2)(\beta_j'')](w)\\
&=&\sum_j \Bigr( \Psi (\theta_1)(\beta_j')\circ \Psi
(\theta_2)(\beta_j'')(w)\\
&\pm &\Psi (\theta_2)(\beta_j'')\circ \Psi
(\theta_1)(\beta_j')(w)\Bigl)\\
&=&\sum_j \beta_j''(b_2)\beta_j'(b_1')w_1'\pm
\beta_j'(b_1)\beta_j''(b_2')w_2'.
\end{eqnarray*}

On the other hand
\begin{eqnarray*}
\Psi ([\theta_1, \theta_2])(\beta )(w)&=&\pm (\beta \otimes
\id_{\Lambda W})[\theta_1,\theta_2](w)\\
&=&\pm (\beta \otimes \id_{\Lambda W})(\theta_1\circ \theta_2 \pm
\theta_2\circ \theta_1 )(w)\\
&=&\pm (\beta \otimes \id_{\Lambda W})(b_2b_1'w_1'\pm b_1b_2'w_2')\\
&=&\pm \beta (b_2b_1')w_1'\pm \beta (b_1b_2')w_2',
\end{eqnarray*}
and both expressions agree, again by the commutativity of the first diagram in (\ref{diag:m and d}).
\end{proof}

We can now deduce our main result:

\begin{proof}[Proof of Theorem \ref{thm:main}]
Let $p \colon E \to B$ be a fibration of simply connected CW complexes with $B$ and the fibre $F$ finite. Let $h \colon B \to \Baut_1(F)$ be the classifying map. We choose a relative minimal model $I \colon \B, d_B \to  \B \otimes \Lambda W, D$ with $\B$ finite-dimensional and the twisted part of $D$ satisfying Lemma \ref{iso}. 
By Theorem \ref{thm:quasi}, the map $\Psi$ is a DG Lie algebra isomorphism.  The result now follows from Propositions \ref{pro:BHMP}, \ref{pro:Sull} and \ref{pro:BFM}.
\end{proof} 
\begin{remark}      Given a fibration $p \colon E \to B$ as hypothesized in Theorem \ref{thm:main},  choose a base-point for $E$ and let $\aut_1^*(p)$ denote the space of base-point preserving fibre homotopy equivalences of $p.$  The above arguments apply to give  a  Quillen model of the form $\overline{\dder}_\B(\B \otimes \Lambda W), \D$ for the space $\Baut^*_1(p).$  Here $$ \overline{\dder}_\B(\B \otimes \Lambda W) = \{ \theta \in \dder_\B(\B \otimes \Lambda W) \mid \theta(w) \in (\B \otimes \Lambda W)_+ \}.$$
The differential $\D$ and bracket $[ \, , \, ]$ are those induced from $\dder_\B(\B \otimes \Lambda W), \D$.  We omit the details. 
\end{remark}

\section{Two problems  on  fibrewise self-equivalences} \label{problems}
    Let $p \colon E \to B$ be a fibration of simply connected CW complexes with $B$ and the fibre $F$ finite.
The assignment $p \mapsto \Baut_1(p)_\Q $ induces a function
\begin{equation}  \label{eq:function} \frac{\{ \hbox{fibrations $p \colon E \to B$ with fibre $F$} \}}{\simeq_{\mathrm{fhe}}}  \,  \longrightarrow  \, \frac{\{ \hbox{classifying spaces $\Baut_1(p)$} \} }{\simeq_\Q} \end{equation}
where $\simeq_{\mathrm{fhe}}$ denotes fibre-homotopy equivalence and $\simeq_\Q$ rational equivalence.
The first general  problem we propose is the following analogue of the gauge group classification problem mentioned in the introduction.
\begin{problem} \label{prob1} Classify the distinct rational homotopy types in  the image of the function \emph{(\ref{eq:function})} for fixed fibre $F$ and base $B$.
\end{problem}

When   the rationalization $Y_\Q$ of a nilpotent space $Y$ is an H-space we say $Y$ is a {\em rational H-space}.   When  $\Baut_1(p)$ is a rational H-space,      the monoid $\aut_1(p)$ is then a double loop-space after rationalization   and, in particular, rationally H-homotopy commutative (cf. \cite{Sch}).  It is natural to consider:

\begin{problem} \label{prob2}  Determine  conditions on $p \colon E \to B$  that ensure $\Baut_1(p)$ is a rational H-space.
\end{problem}

The function space identity for $\Baut_1(p)$  (Proposition \ref{pro:BHMP})  leads to a  class of  familiar examples   for which both problems admit the simplest possible solution. The   basic idea has appeared in many places (see \cite{BHMP, FO, KSS}).    We have:
\begin{theorem} \label{thm:H} Let   $F$ be a CW complex of finite type   and  $B$ a simply connected   finite CW complex. Suppose $\Baut_1(F)$ is a rational H-space.  Then the the classifying spaces $\Baut_1(p)$ for fibrations $p$ with fibre $F$ and base $B$ are all rational H-spaces and pairwise  rationally homotopy equivalent.
\end{theorem}
\begin{proof}
Given an $F$-fibration $p \colon E \to B$, as hypothesized, we have
$$ \Baut_1(p)_\Q \simeq \widetilde{\map}(B, \Baut_1(F); h)_\Q \simeq \widetilde{\map}(B, \Baut_1(F)_\Q; h_\Q)$$
where the first equivalence is (\ref{eq:Baut}) and the second  is \cite[Th.3.11]{HMR}. Here $h_\Q$ denotes $h$ composed with a rationalization $r \colon \Baut_1(F) \to \Baut_1(F)_\Q.$  Now  a rationalized H-space is homotopy equivalent  a loop-space (\cite[Lem.0.1]{Sch}) and so, in particular, admits a homotopy inverse. It follows that  the path-components of $\widetilde{\map}(B, \Baut_1(F)_\Q)$ are each homotopy equivalent to   $\widetilde{\map}(B, \Baut_1(F)_\Q; 0 ),$ the component of the null map.  Since $\map(B, \Baut_1(F)_\Q; 0)$ is   an H-space with pointwise multiplication, so is its universal cover.   \end{proof}
\begin{remarks} \label{rems}
(1)  The famous Halperin conjecture \cite[39.1]{FHT}, affirmed in many cases, implies $\Baut_1(F)$ is a rational H-space whenever  $F$ is an  elliptic CW complex  with $H^*(F; \Q)$ evenly graded. See \cite[Ex.2.6]{FLS} for a discussion in this context.   Theorem \ref{thm:H} applies, for example,  when $F$ has truncated polynomial rational cohomology  or $F = G/H$ for $H \subseteq G$ equal rank, compact, connected Lie pairs \cite{ST}.  

(2)  On the other hand,  if $G$ is a topological group then $\Baut_1(G)$ is rarely a rational H-space: in fact, only when $\pi_*(G) \otimes \Q$  is concentrated in a single degree. Problems  \ref{prob1} and \ref{prob2} are thus already  interesting  when restricted to     principal $G$-bundles over a fixed base $B$.

(3)  A similar argument shows the classifying spaces $B\G(p)$ for gauge groups of $G$-principal bundles $p \colon E \to B$ are pairwise rationally homotopy equivalent and each a rational H-space.  The assumption $B$ simply connected is no longer required since the  identity  $B\G(p) \simeq \map(B, BG; h)$ holds without this restriction.  It is  classical that $BG$ is a rational H-space for $G$ connected.
This result for $B\G(p)$ holds, in fact,  for $B$ any compact metric space \cite[Th.D]{KSS}. Theorem \ref{thm:H}  and (1) above  also extend  to spaces $B$ without CW structure by \cite[Th.2]{SS}.  \end{remarks}

We give a sample of  results on Problems \ref{prob1} and \ref{prob2}   in the some  cases not covered by Theorem \ref{thm:H}. Fix $F$ and $B$ both finite and simply connected.  Let $\B, d_B$ be a finite model for $B$ and $\Lambda W, d_W$ the Sullivan minimal model for $F$. The relative 
minimal models for $F$-fibrations over $B$ are of the form    $$ I \colon\B, d_B \to \B \otimes \Lambda W, D.$$ Recall, the differential $D$ extends $d_B$  and  satisfies the  nilpotence condition whereby $W  = \bigoplus_{i \geq 0} W_{i}$ and $D(W_{i}) \subseteq \B \otimes \Lambda W_{ < i}$ (see \cite[Sec.14]{FHT}).   We are led to consider the possible differentials $D$ on $\B \otimes \Lambda W$.  These differentials  give rise, in turn, to differentials $\D$ for the   graded Lie algebra $\dder_\B(\B \otimes \Lambda W).$
  We will make use of the following consequence  of Theorem \ref{thm:main}.
\begin{corollary} \label{cor}  Let $B$ and $F$ be  finite, simply connected CW complexes. 
\begin{itemize}
\item[(A)] Given $F$-fibrations $p$ and $p'$     over $B$ then $ \Baut_1(p) \simeq_\Q \Baut_1(p')$    if and only if  there is a quasi-isomorphism $$\psi \colon  \dder_\B(\B \otimes \Lambda W), \D \to \dder_\B(\B \otimes \Lambda W), \D'.$$
\item[(B)]  The following are equivalent.
\begin{itemize}
\item[(i)]  $\Baut_1(p)$ is a rational H-space
\item[(ii)] $\dder_\B(\B \otimes \Lambda W), \D$ is quasi-isomorphic to an abelian DG Lie algebra
\item[(iii)]  $\CC^*(\dder_\B(\B \otimes \Lambda W), \D)$ is quasi-isomorphic to a DG algebra with trivial differential.
\end{itemize}\end{itemize}  \qed
\end{corollary}
 \subsection{The  Classification Problem for Principal Bundles.}
  Problem \ref{prob1} is quite complicated for general $F$ and $B$. We restrict attention to principal $G$-bundles $p \colon E \to B$ for fixed $G$ and $B.$
  This ensures that  the differential $D$ is {\em pure}  in the sense that  $D(W) \subseteq \B$   \cite[Sec.15.f]{FHT}.  When $D$ is pure,    the differential  $\D$   reduces to  $\D(\theta) = D \circ \theta$ for $\theta \in \dder_\B(\B \otimes \Lambda W)$.

  Suppose now that   $W \cong \pi_*(G) \otimes \Q$ is finite-dimensional.   Let $n$ denote the maximal nontrivial degree  i.e., $\pi_n(G) \otimes \Q  \neq 0$ and $\pi_k(G) \otimes \Q  = 0$ for $ k > n.$  Here  is one general situation when non-equivalent bundles yield the same rational homotopy type for the classifying space:
 \begin{lemma}  \label{lem:equiv} Let $p$ and $p'$ be principal  $G$-bundles over $B$ with differentials $D$ and $D'$ in their relative minimal models, respectively.
 If $D = D'$ on $W_{ < n}$ where $n$ is the maximal nontrivial degree of $W $  then $\Baut_1(p) \simeq_\Q \Baut_1(p').$
 \end{lemma}
 \begin{proof}
The identity map  for  $\B \otimes \Lambda W $       induces a quasi-isomorphism of DG Lie algebras
 $\dder_\B(\B \otimes \Lambda W), \D \to \dder_B(\B \otimes \Lambda W), \D'$ with these hypotheses.
\end{proof}

  Now  suppose $G$ is a simply connected topological group of the homotopy type of a finite  CW complex.  This class includes the simply connected Lie groups and ensures $W = \pi_*(G) \otimes \Q$ is oddly-graded and finite-dimensional. The principal $G$-bundles  over $B$ are represented by homotopy classes $[B, BG]$. Rationalizing, this set becomes $\Hom^{-1}(W, H^*(B; \Q))$  which corresponds, in turn,  to the set  of pure differentials $D$ on $\B \otimes \Lambda W.$   When $W$ is concentrated in a single degree,    Lemma \ref{lem:equiv} implies the classifying spaces $\Baut_1(p)$ are all of the same rational homotopy type.
When $W$ is concentrated in $2$ or  more degrees the number of rational homotopy types depends on the cohomology of $B$.   Here is a first example.
\begin{proposition}
 The  principal $SU(3)$-bundles $p \colon E \to B$ over a simply connected, finite CW complex $B$    give rise to   either one or two rational homotopy types for $\Baut_1(p)$ depending,  respectively,  on whether   $H^{4}(B; \Q)$ is trivial or not.
 \end{proposition}
\begin{proof}
The relative minimal model for such a bundle $p$ is of the form $$\B, d_B \to \B \otimes \Lambda(s_3, s_5), D$$ where $|s_i| = i.$ If $H^4(B; \Q) = 0$ then $D(s_3) = 0$ and so the result follows from Lemma \ref{lem:equiv}.  

If $H^4(B; \Q) \neq 0$ then we may set $D(s_3) = \chi_4$ for $\langle \chi_4 \rangle \neq 0   \in H^4(B; \Q)$ and $D(s_5)= 0$.  We claim   $$\mathrm{rank}H_2(\dder_\B(\B \otimes \Lambda W) , \D)   = \mathrm{rank}H^{3}(B; \Q)$$ 
while $$\mathrm{rank}H_2( \dder_\B(\B \otimes \Lambda W) , 0) =  \mathrm{rank}H^{3}(B; \Q) +1.$$ 

For the second equation, observe  that, generally,  $$H_n( \dder_\B(\B \otimes \Lambda W) , 0)  = \Hom^n(W, H^*(B; \Q)).$$ For the first equation,  
   suppose the derivation $(s_5, x_3 + qs_3)$ is  a $\D$-cycle. Here we write $(w, \chi)$ for the derivation in $\Der_B(\B \otimes \Lambda W)$ carrying $w$ to $\chi$ and vanishing on a complementary basis for $w \in W$.  
Then  $$ 0 = \D(s_5, x_3 + q s_3) = (s_5, d_B(x_3) + q \chi_4).$$ 
Thus, since $\langle \chi_4 \rangle  \neq 0$, we must have $q = 0$.    

 Finally, if $D'$ is another pure differential with $D'(s_3) = \chi_4'$ another non-bounding cycle then the map sending $\chi_4 \to \chi_4'$ and fixing the other generators of $\B \otimes \Lambda W$
induces a quasi-isomorphim $\dder_\B(\B \otimes \Lambda W), \D \to \dder_B(\B \otimes \Lambda W), \D'.$\end{proof}

   We  give a complete result when $B$ has very simple    rational cohomology. Given a fibration $p \colon E \to B$ with relative minimal model $\B, d_B \to \B \otimes \Lambda W, D,$   and pure differential   $D,$  define a nonnegative integer
$$ n(p) = \left\{ \begin{array}{ll} \mathrm{min} \{ k \mid D(W_k) \neq 0 \}   &  \hbox{for $D(W) \neq 0.$}  \\
n  &  \hbox{for $D(W)= 0.$}  \end{array}
\right.
$$
Example \ref{example1} is contained in the following:
 \begin{proposition} \label{classify} Let $B$ be a simply connected finite CW complex with $H^*(B; \Q)$ generated in a single even degree. Let $G$ be a simply connected topological group with the homotopy type of a finite CW complex. 
 Given two principal $G$-bundles $p, q$ over $B$  we have
 $$ \Baut_1(p) \simeq_\Q \Baut_1(q) \ \ \hbox{ if and only if } \ \ n(p) = n(q).$$
 \end{proposition}
\begin{proof}
Since $B$ is finite,  $H^*(B; \Q)$ is a truncated polynomial algebra. A finite DG algebra model for $B$ is of the form $\B = \Lambda (x)/\langle x^{m+1}  \rangle$ for some $x$ of even degree, say $2k.$
Given a principal $G$-bundle $p \colon E \to B$  , the  relative minimal model for $p$ is thus of the form
$$ \B, 0   \to \B \otimes \Lambda W, D$$
with  $D(x)  = 0$ and $D(w) = \chi(w)$  for $\chi(w)$ a cycle in $\B$ of degree $|w| +1.$
   Note that if  $\chi(w) \neq 0$ then $\chi(w) = qx^j$ for some $j$ and $q \in \Q.$
We may thus write $$W = \Q(w_{1}, \ldots, w_{s}) \oplus W'$$  where  (i) the $w_i$ are of strictly increasing odd degrees $q_i$, (ii) $D(w_i) = x^{n_i}$ for some $n_i \leq n$  and (iii) $D(W') = 0.$   Define a second differential $D'$  by setting $D'(w_1) = x^{n_1},$    $D'(w_i) = 0$ for $i \geq 2$ and $D'(W') =0.$
 Define
$$\psi \colon \B  \otimes \Lambda W, D' \to \B \otimes \Lambda W, D$$    by $\psi(w_1) = w_1$,  $\psi(w_i) = w_i - x^{n_i - n_1}w_1$ for $i \geq 2$ and
$\psi$ the identity on $\Lambda W'$ and $\B$  The map $\psi$ induces a quasi-isomorphim $\dder_\B(\B \otimes \Lambda W), \D' \to \dder_B^*(\B \otimes \Lambda W), \D.$ 

Suppose now that $D$ and $D'$ are differentials on $\B \otimes \Lambda W$     given, on $W$, as $D = (w, x^{k})$ and $D' = (w', x^l)$ for $k,l \leq m.$  Suppose    $|w| \neq |w'|$, say $|w| < |w'|.$  Let $n$ be the maximal nontrivial degree of $W.$ Then we see
$$\mathrm{rank}H_{n-|w|} (\dder_\B(\B \otimes \Lambda W; D)) < \mathrm{rank}H_{n-|w|} (\dder_\B(\B \otimes \Lambda V; D'))$$
since $(w_n, w)$ is a non-bounding $D'$-cycle  but is  not a $D$-cycle for $w_n \in W_n.$
\end{proof}

Applying the functor $\CC^*,$ we can identify the rational homotopy types represented by $\Baut_1(p)$ for fixed $F$ and $B$ in special cases.
We give a representative example:
\begin{example} \label{C*}
We consider the case $F = S^5 \times S^7$ and $B =S^3 \times S^3.$  A finite-dimensional DG algebra model $\B = \Lambda(x_3, y_3)$ in this case.  The relative minimal model for such a fibration is of the form
$$ \B, 0  \to  \B \otimes \Lambda(s_5, s_7), D.$$
  Up to rational fibre homotopy equivalence, there are two possibilities for $D$.  Either $D = 0$ or
$D(s_5) = x_3y_3$ and $D(s_7) = 0$.  In both cases, we have the following basis for $\dder_\B(\B \otimes \Lambda(s_5, s_7)).$
$$ \alpha = (s_7, x_3y_3), \ \ \ \ \beta_1 = (s_5, x_3),  \ \ \ \ \beta_2 = (s_5, y_3),  \ \ \ \  \beta_3 = (s_7, s_5) $$
$$ \gamma_1 = (s_7, x_3), \ \ \ \ \gamma_2 = (s_7, y_3), \ \ \ \ \eta = (s_5, 1), \ \ \ \ \phi = (s_7, 1).$$
Here $|\alpha| = 1, |\beta_i| = 2, |\gamma_i| = 4, |\eta| = 5$ and $|\phi| = 7$.
Thus $\CC^*(\dder_\B(\B \otimes \Lambda(s_5, s_7)), \D)$ is a DG algebra of the form:
$$\Lambda(a,   b_1, b_2, b_3, c_1, c_2,e, f), d$$
where $|a| = 2, |b_i| = 3, |c_i| = 5, |e| = 6$ and $|f| = 8$.

When $D= 0$, the induced differential $\D = 0$ and so the the differential $d$ is given, in this case, by
$$ d(a) = d(b_1) = d(b_2) = d(b_3) = d(e) = 0,  \ \  \ \  d(c_1) = b_3b_1, \ \  \ \  d(c_2)=b_3b_2, \ \ \ \ d(f) =b_3c. $$
It is easy to see that $\Baut_1(p)$ is not a formal space in this case.

When $D \neq 0$ the differential $\D$ is trivial except that
$ \D(\beta_3) = \alpha.$
Removing this linear relation from $\dder_\B(\B \otimes \Lambda(s_5, s_7)), \D$ we see that
$$\Baut_1(p) \simeq K(\Q^2, 3) \times K(\Q^2, 5) \times K(\Q, 6) \times K(\Q, 8)$$
 in this case.
\end{example}

\subsection{When is $\Baut_1(p)$ a rational H-space?}  We give some sample results addressing this question.
By  Theorem \ref{thm:H}, $\Baut_1(p)$ is a rational H-space when  $\Baut_1(F)$ is one.  In addition to the spaces mentioned in Remark \ref{rems} (1), this occurs when $\pi_*(F) \otimes \Q$ is concentrated in a single degree.   When $\pi_*(F) \otimes \Q$ is concentrated in two degrees, we have the following result. Recall a simply connected CW complex $X$  is {\em coformal} if $X$ admits a  Quillen model with vanishing differential, or equivalently, if  the minimal Sullivan model for $X$ has purely quadratic differential.
 \begin{proposition} \label{coformal} Let $p \colon  E \to B$ be a  fibration of simply connected CW complexes with fibre $F$ and base  $B$ finite.  Suppose $\pi_*(F) \otimes \Q$ is concentrated in two degrees.  Then $\Baut_1(p)$ is a coformal space and $\Baut_1(p)$ is a rational H-space if and only if  $\aut_1(p)$ is rationally H-homotopy commutative.
 \end{proposition}
 \begin{proof} In this case, the DG Lie algebra $\dder_\B(\B \otimes \Lambda W), D$ only admits iterated brackets of length  two. It is an easy exercise now to show that a minimal model for  $C^*(\dder_\B(\B \otimes \Lambda W), D)$  has purely quadratic differential.
 In general,  if $X$ is a coformal space then $X$ is a rational H-space if and only if the $\Omega X$ is rationally H-homotopy commutative.
\end{proof}
Focusing   on a concrete example, we have:
\begin{proposition}
Let $p \colon E \to \C P^m$ be a principal $SU(n)$-bundle.  Then $\Baut_1(p)$ is a rational H-space if and only if either $n =  1,  2$  or $n = 3$ and the characteristic class $\chi_1(p) \neq 0 \in H^4(\C P^m; \Q)$.
\end{proposition}
\begin{proof}
When $ n= 1$ the fibre is contratcible and so $\aut_1(p)$ is also.  For $n = 2,$ the result is Theorem \ref{thm:H}, as mentioned above.  For $n=3,$ observe that the induced bracket in  $\dder_\B(\B \otimes \Lambda (s_3, s_5)),  \D$ is nontrivial if and only if the derivation $(s_5, s_3)$ is a $\D$-cycle or, equivalently,    $D(s_3) = 0.$  For $n > 3,$ we invoke the proof of Proposition \ref{classify}.  There we showed that, in the relative minimal model $\B, d_B\to \B \otimes \Lambda(s_3, \ldots, s_{2n-1}),$ we may assume $D(s_{2i +1}) \neq 0$ for  at most   one $i \geq 1$.  Thus $D(s_{2i-1}) = 0$ for some   $2 \leq i < n$.  The bracket $$[(s_{2n-1}, s_{2i-1}), (s_{2i-1}, 1)] = (s_{2n-1}, 1)$$ is then  a non-bounding bracket of $\D$-cycles.  Thus $\pi_*(\aut_1(p)) \otimes \Q$ has a non-trivial Samelson bracket for $n > 3$.     \end{proof}

Finally, we give a connection between Massey products in the rational cohomology of $E$  and nontrivial brackets in the Samelson Lie algebra of $\aut_1(p)$.   Let $p \colon E \to B$ be a principal $G$-bundle where, as usual, $B$  is a finite, simply connected CW complex and $G$ is a simply connected topological group of the homotopy type of a finite CW complex.  Let $j \colon G \to E$ denote the fibre inclusion.  As in Example \ref{example2}, we consider the map
$$J = \mathrm{Hur}_\Q \circ j_\sharp \colon \pi_*(G)   \to H_*(E; \Q)$$
where $\mathrm{Hur}_\Q \colon \pi_*(E) \to H_*(E; \Q)$ is the rational Hurewicz map.   We may identify the image of $J$ in terms of the relative minimal model  $\B, d_B \to \B \otimes \Lambda W, D$  for $p$.  Recall that  $D$  is pure.  Define $$\mathrm{ker}[D_W] = \{ w \in W \mid D(w) \hbox{\, is cohomologous to zero in \,} \B \otimes \Lambda W_{< |w|}, D \}.$$
Note that $\mathrm{ker}[D_W]$ consists of $D$-cycles of the Sullivan model $\B \otimes \Lambda W, D$ for $E$ and thus determines a subspace $K \subseteq H^*(E; \Q)$.
Then we have
\begin{equation} \label{J} \mathrm{Image}(J)  \cong \Hom(K; \Q). \end{equation}
To see this, observe that an element $w \in \mathrm{ker}[D_W]$ corresponds, after a change of basis, to a non-bounding primitive  $D$-cycle  of  generator of the Sullivan model of $E$ and thus to an element in the image of the rational Hurewicz map.
In particular, $J = 0$ if and only if $\mathrm{ker}[D_W] = 0.$
\begin{lemma}  Let $p$ have a relative minimal model as described above. If $\Baut_1(p)$ is a rational H-space then $$\mathrm{Image}(J) \cap H_{< n}(E; \Q) = 0$$ where $n$ is the maximal degree for which $W$ is nontrivial.
\end{lemma}
\begin{proof} Using (\ref{J}),    suppose $D(w) = D(\chi)$ for some $w \in   W_{q}$ with $q < n$ and $\chi \in \B \otimes \Lambda W_{< q}$.  Then we may perform a basis change:  Define $\varphi \colon \B \otimes \Lambda W \to \B \otimes \Lambda W$  by sending $w \to w - \chi$ and fixing the other elements of $W$ and all of $\B$.  Define $D'$ by setting $D'(w) = 0$ and $D' = D$ on a complement  of $\Q(w) \subseteq W$ and on $\B.$  The result is a DG map inducing a quasi-isomorphism $ \dder_\B(\B \otimes \Lambda W), \D' \to \dder_\B(\B \otimes \Lambda W), \D$. In the former DG Lie algebra, we have $$[(w_n, w), (w,1)] = (w_n, 1), $$ a
non-bounding bracket of $\D'$-cycles  for $w_n \in W_n$.
\end{proof}

To discuss Massey products in $H^*(E; \Q)$ we need  $\B \otimes \Lambda W, D$ to be a minimal Sullivan model for $E.$  By \cite[Pro.4.12]{Hal},  this is equivalent to requiring $p \colon E \to B$ to be {\em Whitehead trivial},  i.e., having trivial linking homomorphism  in its long exact rational  homotopy sequence.   In the setting of principal bundles, Whitehead triviality is equivalent to the assumption that the differential $D$ satisfies $$D(W) \subseteq H^{+}(B; \Q) \cdot H^{+}(B; \Q).$$

We recall the characterization of formality in terms of  Massey products in  \cite{DGMS}.  Let $\Lambda V, d$ be a simply connected DG algebra.  Let $C  \subseteq V $ denote the cycles of $V$.  By a {\em Massey product} in $H^*(\Lambda V, d)$ we mean a
nonzero element represented by an cycle in $I(V') \subseteq \Lambda V$ where $V'$ is some vector space complement to $C \subseteq V$ and $I(V') \subseteq \Lambda V$
is  the ideal generated by $V'$. By \cite[Lem.4.3]{DGMS}, $\Lambda V, d$ is formal if and only if all Massey products vanish.
The following result implies Example \ref{example2}.

\begin{theorem}   Let   $B$ be a finite simply connected CW complex and $G$ a simply connected topological group of the homotopy type of a finite CW complex.  Let $ p \colon E \to B$ be a Whitehead trivial, principal $G$-bundle.    Let $n$ be the  maximal degree for which $\pi_n(G) \otimes \Q \neq 0.$ Suppose  (i)  $J  \colon \pi_*(G) \to H_*(E; \Q)$ vanishes in degrees $ < n$ and  (ii)  $H^{*}(E; \Q)$ has no Massey products of degree $ < n$. Then $\Baut_1(p)$ is a rational H-space. \end{theorem}
\begin{proof}
Since  $J$    vanishes in degrees $ < n,$ in the minimal Sullivan model $\B \otimes \Lambda W, D$ for $E$   we may take $W \subseteq V'$.      Suppose     $\theta$ is a cycle  in $\dder^q_\B(\B \otimes \Lambda W), \D$. Given $w \in W$, write  $\theta(w) =  z + m$ where $m \in I(V')$ and $z \in \Lambda C.$  Since $\theta$ is a $\D$-cycle, $D(m) = 0$.  Thus $m = D(y)$ for some $y \in \B \otimes \Lambda W$.    We then have $(\theta - \D(w, y))(w) \subseteq \B.$   Continuing in this way for each element in a basis for $W,$ we obtain that
$\theta$ is  homologous to a $\D$-cycle $\theta'$ satisfying $\theta'(W) \subseteq \B.$
It follows that the inclusion $$\lin(W, \B), 0 \hookrightarrow \dder_\B(\B \otimes \Lambda W), D$$ is a quasi-isomorphism.  Since $\lin(W, \B), 0$ is an abelian DG Lie algebra, the result follows from Corollary \ref{cor}.B. 
\end{proof}

\end{document}